\documentclass[11pt]{amsart}
\usepackage{amsmath,amssymb,amsthm,amscd}
\usepackage{graphicx,epstopdf}
\usepackage{palatino}
\usepackage{url}
\usepackage{xy}
\xyoption{all}

\newtheorem{theorem}{Theorem}[section]
\newtheorem{lemma}[theorem]{Lemma}

\newtheorem{definition}[theorem]{Definition}
\newtheorem{proposition}[theorem]{Proposition}
\newtheorem{corollary}[theorem]{Corollary}

\newcommand{\eg}{{\em e.g.}}

\newcommand{\real}{\mathbb{R}}       
\newcommand{\euc}{\mathbb{E}}        
\newcommand{\Rips}{{\mathcal R}}        
\newcommand{\Shadow}{{\mathcal S}}      
\newcommand{\proj}{{p}}                 
\newcommand{\shadowface}{{\Psi}}        
\newcommand{\Cech}{{\mathcal C}}        
\newcommand{\df}[1]{\textbf{{#1}}}          
\newcommand{\spanned}[1]{{\left\langle{#1}\right\rangle}} 
\newcommand{\abs}[1]{\mathopen|#1\mathclose|}       

\graphicspath{{Figures/}}

\begin{document}

\title{RIPS COMPLEXES OF PLANAR POINT SETS}

\author{Erin W. Chambers}
    \address{Department of Computer Science, University of Illinois,
    Urbana-Champaign}
    \email{erinwolf@uiuc.edu}
    \thanks{EWC supported by NSF MSPA-MCS \# 0528086.}

\author{Vin de Silva}
    \address{Department of Mathematics, Pomona College, Claremont CA.}
    \email{vin.desilva@pomona.edu}
    \thanks{VdS supported by DARPA SPA \# 30759.}

\author{Jeff Erickson}
    \address{Department of Computer Science, University of Illinois,
    Urbana-Champaign}
    \email{jeffe@cs.uiuc.edu}
    \thanks{JE supported by NSF MSPA-MCS \# 0528086.}

\author{Robert Ghrist}
        \address{Department of Mathematics and Coordinated Science
        Laboratory, University of Illinois, Urbana-Champaign.}
    \email{ghrist@math.uiuc.edu}
    \thanks{RG supported by DARPA SToMP \# HR0011-07-1-0002 and
    NSF MSPA-MCS \# 0528086.}


\begin{abstract}
Fix a finite set of points in Euclidean $n$-space $\euc^n$, thought
of as a point-cloud sampling of a certain domain $D\subset\euc^n$.
The Rips complex is a combinatorial simplicial complex based on
proximity of neighbors that serves as an easily-computed but
high-dimensional approximation to the homotopy type of $D$. There is
a natural ``shadow'' projection map from the Rips complex to
$\euc^n$ that has as its image a more accurate $n$-dimensional
approximation to the homotopy type of $D$.

We demonstrate that this projection map is 1-connected for the
planar case $n=2$. That is, for planar domains, the Rips complex
accurately captures connectivity and fundamental group data. This
implies that the fundamental group of a Rips complex for a planar
point set is a free group. We show that, in contrast, introducing
even a small amount of uncertainty in proximity detection leads to
`quasi'-Rips complexes with nearly arbitrary fundamental groups.
This topological noise can be mitigated by examining a pair of
quasi-Rips complexes and using ideas from persistent topology.
Finally, we show that the projection map does not preserve
higher-order topological data for planar sets, nor does it preserve
fundamental group data for point sets in dimension larger than
three.
\end{abstract}

\maketitle
\section{Introduction}
\label{sec:intro}

Given a set $X$ of points in Euclidean space $\euc^n$, the
\df{Vietoris-Rips complex} $\Rips_\epsilon(X)$ is the abstract
simplicial complex whose $k$-simplices are determined by subsets of
$k+1$ points in $X$ with diameter at most $\epsilon$. For
simplicity, we set $\epsilon=1$ and write $\Rips:=\Rips_1(X)$ for
the remainder of the paper, with the exception of \S\ref{sec:quasi}.
For brevity (and to conform to typical usage), we refer to $\Rips$
as the \df{Rips complex}. The Rips complex is an example of a
\df{flag complex} --- the maximal simplicial complex with a given
1-skeleton.

The Rips complex was used by Vietoris \cite{Vietoris} in the early
days of homology theory, as a means of creating finite simplicial
models of metric spaces. Within the past two decades, the Rips
complex has been utilized frequently in geometric group theory
\cite{Gromov} as a means of building simplicial models for group
actions. Most recently, Rips complexes have been used heavily in
computational topology, as a simplicial model for point-cloud data
\cite{CCD,CIDZ,CZCG,DC}, and as simplicial completions of
communication links in sensor networks \cite{DG:controlled,
DG:persistence,Jad}.

The utility of Rips complexes in computational topology stems from
the ability of a Rips complex to approximate the topology of a cloud
of points. We make this notion more specific. To a collection of
points, one can assign a different simplicial model called the
\v{C}ech complex that accurately captures the homotopy type of the
cover of these points by balls. Formally, given a set $X$ of points
in some Euclidean space $\euc^n$, the \df{\v{C}ech complex}
$\Cech_\epsilon(X)$ is the abstract simplicial complex where a
subset of $k+1$ points in $X$ determines a $k$-simplex if and only
if they lie in a ball of radius $\epsilon/2$. The \v{C}ech complex
is equivalently the nerve of the set of closed balls of radius
$\epsilon/2$ centered at points in $X$. The \df{\v{C}ech theorem}
(or \df{Nerve lemma}, see, \eg, \cite{Leray}) states that
$\Cech_\epsilon(X)$ has the homotopy type of the union of these
balls.  Thus, the \v{C}ech complex is an appropriate simplicial
model for the topology of the point cloud (where the parameter
$\epsilon$ is a variable).

There is a price for the high topological fidelity of a \v{C}ech
complex. Given the point set, it is nontrivial to compute and store
the simplices of the \v{C}ech complex. The virtue of a Rips complex
is that it is determined completely by its 1-skeleton --- the
proximity graph of the points. (This is particularly useful in the
setting of {\em ad hoc} wireless networks, where the hardware
establishes communication links based, ideally, on proximity of
nodes.) The penalty for this simplicity is that it is not
immediately clear what is encoded in the homotopy type of $\Rips$.
Like the \v{C}ech complex, it is not generally a subcomplex of its
host Euclidean space~$\euc^n$, and, unlike the \v{C}ech complex it
need not behave like an $n$-dimensional space at all: $\Rips$ may
have nontrivial topological invariants (homotopy or homology groups)
of dimension $n$ and above.

The disadvantage of both \v{C}ech and Rips complexes are in their
rigid cut-offs as a function of distance between points. Arbitrarily
small perturbations in the locations of the points can have dramatic
effects on the topology of the associated simplicial complexes.
Researchers in sensor networks are acutely aware of this limitation,
given the amount of uncertainty and fluctuation in wireless
networks. To account for this, several researchers in sensor
networks have used a notion of a distance-based communication graph
with a region of uncertain edges \cite{Barriere,Kuhn}. This
motivates the following construction.

Fix an open \df{uncertainty interval} $(\epsilon,\epsilon')$ which
encodes connection errors as a function of distance. For all nodes
of distance $\leq\epsilon$, there is an edge, and for all nodes of
distance $\geq\epsilon'$, no edge exists. For nodes of distance
within $(\epsilon,\epsilon')$, a communication link may or may not
exist. A \df{quasi-Rips complex} with uncertainty interval
$(\epsilon, \epsilon')$ is the simplicial flag complex of such a
graph. We note that this does not model {\em temporal} uncertainty,
merely spatial.

A completely different model of simplicial complexes associated to a
point cloud comes from considering shadows. Any abstract simplicial
complex with vertices indexed by geometric points in $\euc^n$ (\eg,
a Rips, \v{C}ech, or quasi-Rips complex) has a canonical \df{shadow}
in $\euc^n$, which strikes a balance between computability and
topological faithfulness. For, say, a Rips complex, the canonical
\df{projection} $\proj\colon\Rips\to\euc^n$ is the well-defined
function that maps each simplex in $\Rips$ affinely onto the convex
hull of its vertices in $\euc^n$.  This projection map is continuous
and piecewise-linear. The shadow $\Shadow$ is the image
$\proj(\Rips)$ of this projection map.

%
\begin{figure}[hbt]
\centerline{\includegraphics[width=5.0in]{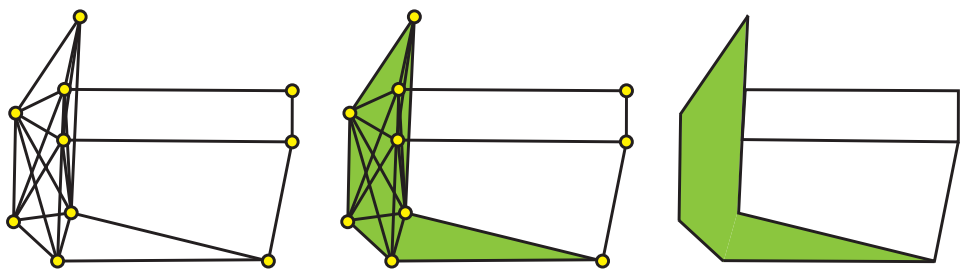}} \label{fig:intro}
\caption{A connectivity graph in the plane [left] determines a
5-dimensional (Vietoris-) Rips complex [center] and its
2-dimensional projected shadow [right].}
\end{figure}
%

This paper studies the topological faithfulness of the projection
map $\proj$ (see Figure \ref{fig:intro}). Specifically, we look at
the connectivity of $\proj$. Recall that a topological map $f:X\to
Y$ is \df{$k$-connected} if the induced homomorphisms on homotopy
groups $\proj_*:\pi_i(X)\to\pi_i(Y)$ are isomorphisms for all $0\leq
i\leq k$: \eg, a 1-connected map preserves path-connectivity and
fundamental group data.

We can now list the principal results of the paper, ordered as they
appear in the following sections.

\begin{enumerate}
\item
For any set of points in $\euc^2$,
$\pi_1(\proj)\colon\pi_1(\Rips)\to\pi_1(\Shadow)$ is an isomorphism.
\item
The fundamental group of any planar Rips complex is free.
\item
Given any finitely presented group $G$, there exists a quasi-Rips
complex $\Rips_Q$ with arbitrarily small uncertainty interval such
that $\pi_1(\Rips_Q)$ is a free extension of $G$.
\item
Given a pair of quasi-Rips complexes $\Rips_Q$, $\Rips_{Q'}$ with
disjoint uncertainty intervals, the image of
$\iota_*:\pi_1(\Rips_Q)\to\pi_1(\Rips_{Q'})$ is free.
\item
The projection map $\proj$ on $\real^n$ is always $k$-connected for
$k=0$ or $n=1$. For all other cases except $(k,n)=(1,2)$ and,
perhaps, $(1,3)$, $k$-connectivity fails on $\real^n$ (see Figure
\ref{fig:knrips}).
\end{enumerate}

\section{Planar Rips complexes and their shadows}
\label{sec:background}

In this section, we restrict attention to the 2-dimensional case.

\subsection{The shadow complex}

The shadow $\Shadow$ is a polyhedral subset of the plane. By
Carath\'eodory's theorem \cite{Helly}, $\Shadow$ is the projection of
the 2-skeleton of~$\Rips$.  Since the vertices of~$\Rips$ are distinct
points in the plane, it follows that distinct edges of~$\Rips$ have
distinct images under~$\proj$, and these are nondegenerate.
Informally we will identify vertices and edges
of~$\Rips$ with their images under~$\proj$. On the other hand,
$\proj$ may be degenerate on 2-simplices.

We can canonically decompose $\Shadow$ into a 2-dimensional
\df{shadow complex} as follows:
\begin{itemize}
\item
A \df{shadow vertex} is either a vertex of $\Rips$
or a point of transverse intersection of two edges of~$\Rips$.
We write $\Shadow^{(0)}$ for the set of shadow vertices.

\item
A \df{shadow edge} is the closure of any component
of $\proj(\Rips^{(1)})\setminus \Shadow^{(0)}$.  Each shadow edge
is a maximal line segment contained in a Rips edge, with no
shadow vertices in its interior.  We write $\Shadow^{(1)}$
for the union of all shadow vertices and shadow edges.

\item
Finally, a \df{shadow face} is the closure of any bounded
component of~$\euc^2 \setminus \Shadow^{(1)}$.
\end{itemize}

The fundamental group $\pi_1(\Shadow)$ may now be described in terms
of combinatorial paths of shadow edges modulo homotopy across shadow
faces, whereas $\pi_1(\Rips)$ may be described in terms of
combinatorial paths of Rips edges modulo homotopy across Rips faces.
This description opens the door to combinatorial methods in the
proof that $\pi_1(\proj)$ is an isomorphism.

\subsection{Technical Lemmas} \label{sec:lemmas}

Theorem \ref{thm:main} will follow from reduction to three special
cases. We prove these cases in this subsection. We use the following
notation.  Simplices of a Rips complex will be specified by square
braces, e.g., $[ABC]$. Images in the shadow complex will be denoted
without adornment, e.g., $ABC$. The Euclidean length of an edge $AB$
will be denoted $\abs{AB}$. Braces $\spanned{\cdot}$ will be used to
denote the span in $\Rips$: the smallest subcomplex containing a
given set of vertices, e.g., $\spanned{ABCD}$.

The following propositions address the three special cases of
Theorem~\ref{thm:main} which are used to prove the theorem.  Certain
induced subcomplexes of~$\Rips$ are shown to be simply connected. In
the first two cases, it is helpful to establish the stronger
conclusion that these subcomplexes are \df{cones}: all maximal
simplices share a common vertex, called the \df{apex}. The first of
these cases is trivial and well-known (\emph{viz.},
\cite{DG:controlled,FGG04}).

\begin{proposition}
\label{prop:abyz}
Let $\Rips=\spanned{ABYZ}$ be a Rips complex containing simplices
$[AB]$ and $[YZ]$ whose images in $\Shadow$ intersect.  Then $\Rips$
is a cone.
\end{proposition}

\begin{proof}
Let $x$ be the common point of $AB$ and $YZ$.  Each edge is split at
$x$ into two pieces, at most one of which can have length more than
one-half. The triangle inequality implies that the shortest of these
four half-edges must have its endpoint within unit distance of both
endpoints of the traversing edge, thus yielding a 2-simplex in
$\Rips$.
\end{proof}

\begin{proposition}
\label{prop:abxyz}
Let $\Rips=\spanned{ABXYZ}$ be a Rips complex containing simplices
$[AB]$ and $[XYZ]$ whose images in $\Shadow$ intersect.
Then $\Rips$ is a cone.
\end{proposition}

\begin{proof}
The edge $AB$ intersects the triangle $XYZ$. If $AB$ intersects only
one edge of $XYZ$, then one vertex of $AB$ (say, $A$) lies within
$XYZ$ and cones off a 3-simplex $[AXYZ]$ in $\Rips$. Therefore,
without a loss of generality we may assume $AB$ crosses $ZY$ and
$ZX$.

By Proposition~\ref{prop:abyz}, the subcomplexes $\spanned{ABXZ}$
and $\spanned{ABYZ}$ are cones.  If these two cones have the same
apex, then the entire Rips complex $\Rips$ is a cone with that apex.
Similarly, if either apex lies inside the image triangle $XYZ$, then
$\Rips$ is a cone with that apex.  The only remaining possibility is that
$A$ is the apex of one subcomplex and $B$ is the apex of the other; in this
case, $\Rips$ is a cone over $Z$, since both $A$ and $B$ are connected
to $Z$.
\end{proof}

%
\begin{figure}[hbt]
\centerline{\includegraphics[width=2.25in]{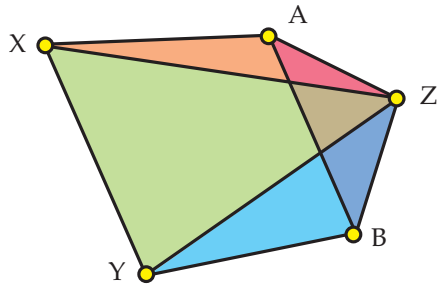}} \caption{The last
case of Proposition \ref{prop:abxyz}.}
\end{figure}
%

\begin{proposition}
\label{prop:abcdxyz}
Let $\Rips=\spanned{ABCDXYZ}$ be a Rips complex containing simplices
$[AB]$, $[CD]$ and $[XYZ]$ whose images in $\Shadow$ meet in a common
point.  Moreover, assume that none of $A, B, C, D$ lies in the interior of $XYZ$.
Then $\pi_1(\Rips)$ is trivial.
\end{proposition}
%
\begin{figure}[hbt]
\centerline{\includegraphics[width=2.5in]{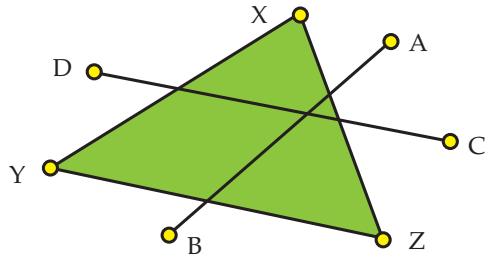}}
\label{fig:abcdxyz} \caption{The setup for Proposition
\ref{prop:abcdxyz}.}
\end{figure}
%

To prove Proposition~\ref{prop:abcdxyz}, we use two further
geometric lemmas.

\begin{lemma}
\label{lem:bxyz}
Let $\Rips = \spanned{BXYZ}$ be a Rips complex containing simplex $[XYZ]$.
If $M$ is a point in $XYZ$ such that $\abs{BM} \leq \frac{1}{2}$, then
$\Rips$ contains at least one of the edges $[BX]$, $[BY]$, $[BZ]$.
\end{lemma}

\begin{proof}
If $B$ lies in $XYZ$ then all three edges belong to~$\Rips$.
Otherwise, $BM$ meets the boundary of $XYZ$ at a point~$M'$. We may
assume that $M'$ lies on $XY$, with $\abs{M'X} \leq \abs{M'Y}$.
Then $\abs{BX} \leq \abs{BM'} + \abs{M'X} \leq \frac{1}{2} + \frac{1}{2} = 1$.
\end{proof}

\begin{lemma}
\label{lem:abcxyz}
Let $\Rips=\spanned{ABCXYZ}$ be a Rips complex containing simplices
$[ABC]$ and $[XYZ]$.  Suppose that $AB$ intersects $XYZ$ but $BC$
and $AC$ do not.  Then $\Rips$ is a cone.
\end{lemma}

\begin{proof}
The hypotheses of the lemma imply that at least one of the points $X$, $Y$,
or $Z$ lies in the interior of $ABC$.  $\Rips$ is a cone on this point.
\end{proof}

\begin{proof}[Proof of Proposition~\ref{prop:abcdxyz}]
We argue by exhaustive case analysis that $\Rips$ contains no
minimal non-contractible cycle.

Suppose $\gamma$ is a minimal non-contractible cycle in~$\Rips$.
Because $\Rips$ is a flag complex, $\gamma$ must consist of at least
four Rips edges.  Our previous Propositions imply that this cycle intersects
each simplex $[AB]$, $[CD]$, and $[XYZ]$ at least once.  By minimality,
$\gamma$ contains at most one edge of $[XYZ]$.  Thus, we may assume
without loss of generality (by relabeling if necessary) that $\gamma$ is
of the form $A(B)C(D)X(Y)$ where $(\cdot)$ denotes an optional letter.

\medskip
\textit{Claim~1: In a minimal cycle, the subwords $ABCD$, $CDXY$,
$XYAB$ are impossible.}  Proposition~\ref{prop:abyz} (in the
first case) and Proposition~\ref{prop:abxyz} (in the last two
cases) imply that the subpaths corresponding to these subwords are
homotopic (relative to endpoints) within a cone subcomplex to a path
with at most two edges, contradicting the minimality of $\gamma$.

Claim~1 implies that that there is at most (i.e.\ exactly) one
optional letter. This leaves three possible minimal non-contractible
cycles: $ACXY$, $ABCX$, and $ACDX$.  The last two cases differ only
by relabeling, so it suffices to consider only $ACXY$ and $ABCX$.

\medskip
\textit{Claim~2: $ACXY$ is impossible.} Suppose $ACXY$ is a cycle
in $\Rips$.  If $AC$ meets $XYZ$ then Proposition~\ref{prop:abxyz}
implies that $\spanned{ACXYZ}$ is a cone, so $ACXY$ is contractible.
Thus, we can assume that $AC$ does not meet $XYZ$.

By Proposition~\ref{prop:abyz}, either $[BC]$ or $[AD]$ is a Rips
edge. Without loss of generality, assume $[BC]$ is a Rips edge; then
$[ABC]$ is a Rips triangle.  If $BC$ does not meet $XYZ$, then
Lemma~\ref{lem:abcxyz} implies that $\spanned{ABCXYZ}$ is a cone,
and hence that $ACXY$ is contractible.  Thus we can assume that $BC$
intersects $XYZ$.

Proposition~\ref{prop:abxyz} now implies that both $\spanned{ABXYZ}$
and $\spanned{BCXYZ}$ are cones.  If any of the segments $[BX]$,
$[BY]$, $[BZ]$ is a Rips edge, then the cycle $ACXY$ is homotopic to
the sum of two cycles, contained respectively in the cones
$\spanned{ABXYZ}$ and $\spanned{BCXYZ}$, and hence is contractible.
See Figure~\ref{fig:acxysplit}(a).

%
\begin{figure}[hbt]
\begin{center}
\includegraphics[width=4.5in]{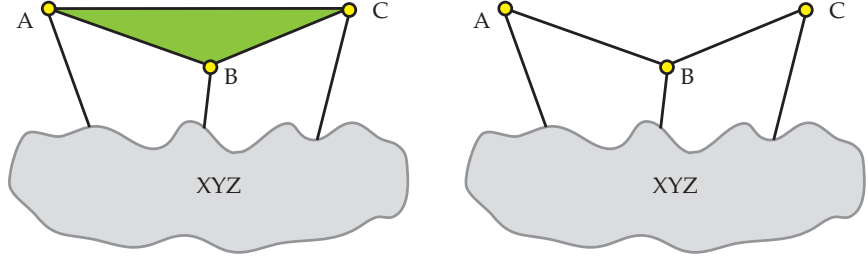}
\caption{$ACXY$ (left), or $ABCX$ (right), splits into two cycles in
the presence of $[BX]$, $[BY]$, or $[BZ]$.} \label{fig:acxysplit}
\end{center}
\end{figure}

We can therefore assume that none of the segments $[BX]$, $[BY]$,
$[BZ]$ is a Rips edge.  In this case, the apex of $\spanned{ABXYZ}$
must be $A$. In particular, the diagonal $[AX]$ of the cycle $ACXY$
belongs to~$\Rips$, and so $ACXY$ is contractible.  This completes
the proof of Claim~2.

\medskip
\textit{Claim~3: $ABCX$ is impossible.}
Suppose $ABCX$ is a cycle in $\Rips$.  If either $[AC]$ or $[BX]$ is a
Rips edge, then $ABCX$ is trivially contractible.
Moreover, if either $[BY]$ or $[BZ]$ is a Rips edge, then the
cycle $ABCX$ reduces to the sum of two cycles, as in
Figure~\ref{fig:acxysplit}(b). The
left cycle is contractible by Proposition~\ref{prop:abxyz}, and
the right cycle is contractible by Claim~2 (suitably
relabeled), so $ABCX$ is contractible in that case too.
We can therefore assume that none of the segments $[AC]$,
$[BX]$, $[BY]$, or $[BZ]$ is a Rips edge.

Now let $M$ be a common point of intersection of $AB$, $CD$, and
$XYZ$.  Lemma~\ref{lem:bxyz} implies that $\abs{BM} > \frac{1}{2}$,
and so $\abs{AM} = \abs{AB} - \abs{BM} \leq \frac{1}{2}$.  Since
$\abs{AC}>1$, we have $\abs{CM} = \abs{AC} - \abs{AM} > \frac{1}{2}$,
and so $\abs{DM} = \abs{CD} - \abs{CM}\leq \frac{1}{2}$.  These
inequalities imply that $\abs{AD} \leq \abs{AM} + \abs{DM} \leq 1$,
so $[AD]$ is a Rips edge.

It follows that $\Rips$ contains the cycle $ADCX$.  This cycle is
homotopic to $ABCX$, since $\spanned{ABCD}$ is a cone by Proposition
\ref{prop:abyz}.  Lemma~\ref{lem:bxyz} implies that at least one of
the segments $[DX]$, $[DY]$, $[DZ]$ must be  a Rips edge.  Arguing
as before, with $D$ in place of $B$, we conclude that $ADCX$, and
thus $ABCX$, is contractible.  This completes the proof of Claim~3.
\end{proof}


\subsection{Lifting Paths via Chaining}

For any path $\alpha$ in $\Rips^{(1)}$, the projection
$\proj(\alpha)$ is a path in $\Shadow^{(1)}$, but not every shadow
path is the projection of a Rips path. Every oriented shadow edge in
$\Shadow$ is covered by one or more oriented edges in $\Rips$. Thus
to every path in $\Shadow^{(1)}$ can be associated a sequence of
oriented edges in $\Rips$. These edges do not necessarily form a
path, but projections of consecutive Rips edges necessarily
intersect at a shadow vertex.

\begin{definition}
\label{def:chainingseq}
Let $[AB]$ and $[CD]$ be oriented Rips edges induced by consecutive
edges in some shadow path. A \df{chaining sequence} is a path from
$A$ to $D$ in the subcomplex $\spanned{ABCD}$ which begins with the
edge $AB$ and ends with the edge $CD$.
\end{definition}

If we concatenate chaining sequences of shadow edges in $\Shadow$ by
identifying the Rips edges in the beginning and end of adjacent
lifting sequences, we obtain a \df{lift} of the shadow path to
$\Rips$. For any shadow path $\alpha$ in $\Shadow$, we let
$\widehat\alpha$ denote a lift of $\alpha$ to the Rips complex by
means of chaining sequences. Note that the lift of a shadow path is
not a true lift with respect to the projection map $\proj$
--- the endpoints, for example, may differ.

\begin{lemma}
\label{lem:uniquelift}
For any path $\alpha$ in $\Shadow^{(1)}$, any two lifts of $\alpha$
to $\Rips$ with the same endpoints are homotopic in $\Rips$ rel
endpoints.
\end{lemma}

\begin{proof}
Let $\sigma$ and $\tau$ be consecutive shadow edges in $\alpha$, and
let $[AB]$ and $[CD]$ be Rips edges such that $\sigma\subseteq AB$
and $\tau\subseteq CD$.  Proposition \ref{prop:abyz} implies that
all chaining sequences from $A$ to $D$ are homotopic rel endpoints
in $\spanned{ABCD}$, and thus in $\Rips$.  If every shadow edge in
$\alpha$ lifts to a unique Rips edge, the proof is complete.

On the other hand, suppose $\tau\subseteq CD\cap C'D'$ for some Rips
edge $[C'D']$ that overlaps $[CD]$. Proposition~\ref{prop:abyz}
implies that both $[CC']$ and $[DD']$ are Rips edges. Moreover,
since $AB$ intersects $CD\cap C'D'$, any chaining sequence from $A$
to $D$ is homotopic rel endpoints in $\Rips$ to any chaining
sequence from $A$ to $D'$ followed by $[D'D]$.  Thus, concatenation
of chaining sequences is not dependent on uniqueness of edge lifts.
\end{proof}

We next show that the projection of a lift of any two consecutive
shadow edges is homotopic to the original edges.

\begin{lemma}
\label{lem:twoedgesurj}
For any two adjacent shadow edges $wx$ and $xy$, where $AB$ and $CD$
are Rips edges with $wx \subseteq AB$ and $xy \subseteq CD$,
$\proj(\widehat{wx \cdot xy})$ is homotopic rel endpoints to the
path $ABxCD$ in $\Shadow$.
\end{lemma}

\begin{figure}
\begin{center}
\includegraphics[height=2in]{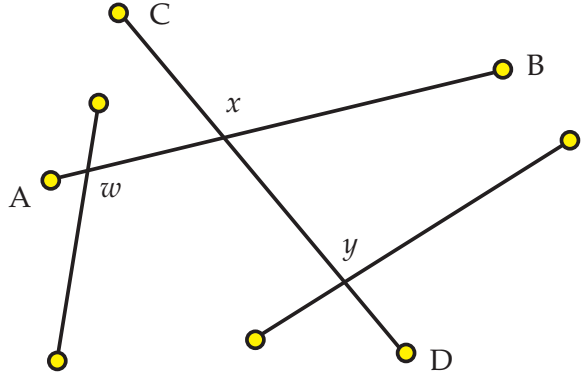}
\caption{The setting for Lemma~\ref{lem:twoedgesurj}}
\label{fig:Surj}
\end{center}
\end{figure}

\begin{proof}
Consider the possible chaining sequences from $A$ to $D$ for $wx
\cdot xy$. Either $BC$ or $AD$ must exist in $\Rips$ by
Proposition~\ref{prop:abyz}.

Suppose $BC$ exists.  By Lemma~\ref{lem:uniquelift}, the chaining
sequence is the Rips path $ABCD$ (up to homotopy rel endpoints).
Either the triangle $[ABC]$ or the triangle $[BCD]$ exists in
$\Rips$ by Proposition~\ref{prop:abyz}, so the triangle $BCx$ is in
shadow. This gives that $ABCD \simeq AxD \simeq ABxCD$ in $\Shadow$.

If $BC$ is not a Rips edge, then $AD$ must be a Rips edges. By
Lemma~\ref{lem:uniquelift}, the chaining sequence is the Rips path
$ABADCD$ (up to homotopy rel endpoints).  Either the triangle
$[ACD]$ or the triangle $[ABD]$ exists in $\Rips$ by
Proposition~\ref{prop:abyz}. Therefore, $ADx$ lies in the shadow, so
we get $ABADCD \simeq ABxCD$ in $\Shadow$.
\end{proof}

\begin{lemma}
\label{lem:lifting}
For any lift $\widehat\alpha$ of any shadow path $\alpha$ with
endpoints in $\proj(\Rips^{(0)})$, we have $\proj(\widehat \alpha)
\simeq \alpha$ rel endpoints.
\end{lemma}
\begin{proof}

For each pair of edges consecutive shadow edges $wx$ and $xy$ in
$\alpha$, where $wx \subseteq AB$, $xy \subseteq CD$, and $AB$ and
$CD$ are Rips edges, Lemma~\ref{lem:twoedgesurj} says that the
projection of their lifting sequence deforms back to $ABxCD$.  Every
adjacent pair of chaining sequences can still be identified along
common edges, since each ends with the first edge in the next one
along $\alpha$. The projection is homotopic rel endpoints to the
original path $\alpha$ except for spikes of the form $xB$ and $xC$
at each shadow junction, which can be deformation retracted, giving
$\proj(\widehat \alpha) \simeq \alpha$.
\end{proof}

\section{$1$-connectivity on $\real^2$}
\label{sec:reduction}

The following is the main theorem of this paper.

\begin{theorem}
\label{thm:main} For any set of points in $\euc^2$,
$\pi_1(\proj)\colon\pi_1(\Rips)\to\pi_1(\Shadow)$ is an isomorphism.
\end{theorem}

\begin{proof} 
Assume that all $\pi_1$ computations are performed with a basepoint
in $\proj(\Rips^{(0)})$, to remove ambiguity of endpoints in lifts
of shadow paths to $\Rips$. Surjectivity of $\proj$ on $\pi_1$
follows from Lemma \ref{lem:lifting} and the fact that any loop in
$\Shadow$ is homotopic to a loop of shadow edges thanks to the cell
structure of $\Shadow$.

To prove injectivity, note that any contractible cycle in $\Shadow$
is expressible as a concatenation of boundary loops of shadow faces
(conjugated to the basepoint). Thanks to Lemma \ref{lem:lifting},
injectivity of $\pi_1(\proj)$ will follow by showing that the
boundary of any shadow face lifts to a contractible loop in $\Rips$.
Consider therefore a shadow face $\shadowface$ contained in the
projection of a Rips 2-simplex $[XYZ]$, and choose $[XYZ]$ to be
minimal in the partial order of such 2-simplices generated by
inclusion on the projections.

Write $\partial \shadowface$ as $\alpha_1 \cdot \alpha_2 \cdots
\alpha_n$, where the $\alpha_i$ are the shadow edges, and let $[A_i
B_i]$ be a sequence of directed Rips edges with $\alpha_i \subseteq
[A_i B_i]$. Neither the $A_i$ nor the $B_i$ project to the interior
of $XYZ$ (see Figure \ref{fig:shadowface}); if any Rips vertex $W$
did so, the edges $[XW]$, $[YW]$ and $[ZW]$ would exist in $\Rips$.
As $\shadowface$ cannot be split by the image of any of these three
edges, it must be contained in the projected image of a Rips
2-simplex, say $[XYW]$, whose image lies within that of $[XYZ]$,
contradicting the minimality assumption on $[XYZ]$. The hypotheses
of Proposition \ref{prop:abcdxyz} thus apply to $[XYZ]$ and the
consecutive edges $[A_iB_i]$, $[A_{i+1}B_{i+1}]$, and each complex
$\spanned{A_iB_iA_{i+1}B_{i+1}XYZ}$ is simply connected.

\begin{figure}[hbt]
\begin{center}
\includegraphics[width=3.0in]{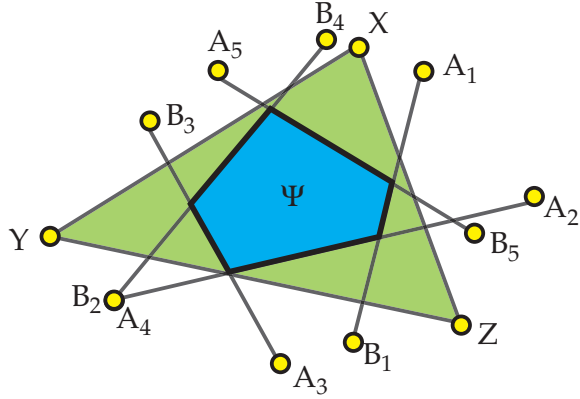}
\caption{The boundary of a shadow face $\shadowface$ within $XYZ$ is
determined by Rips edges $[A_iB_i]$ whose projected endpoints lie outside
$XYZ$.} \label{fig:shadowface}
\end{center}
\end{figure}

Fix the vertex $X$ as a basepoint and fix a sequence of edge paths
$\beta_i$ in $\spanned{A_iB_iXYZ}$ from $X$ to $A_i$. Such paths
exist and are unique up to homotopy since (by
Proposition~\ref{prop:abxyz}) $\spanned{A_iB_iXYZ}$ is a cone. We
decompose $\widehat{\partial\shadowface}$ into loops
$\gamma_1\cdots\gamma_n$, where $\gamma_i$ is the loop with
basepoint $X$ given by
\[
    \gamma_i
    =
    \beta_i
        \cdot
    \widehat{(\alpha_i \cdot \alpha_{i+1})}
        \cdot
    [B_{i+1} A_{i+1}]
        \cdot
    \beta_{i+1}^{-1}
\]
where all indices are computed modulo $n$. By
Proposition~\ref{prop:abcdxyz}, each of these loops $\gamma_i$ is
contractible; hence, so is $\widehat\shadowface$.
\end{proof}

\begin{corollary}
\label{cor:torsion}
The fundamental group of a Rips complex of a planar point set is
free.
\end{corollary}

\section{Quasi Rips complexes and shadows}
\label{sec:quasi}

We observe that Theorem \ref{thm:main} fails for quasi-Rips
complexes, even for those with arbitrarily small uncertainty
intervals. The failure of Proposition \ref{prop:abyz} in the
quasi-Rips case makes it a simple exercise for the reader to
generate examples of a quasi-Rips complexes which are
simply-connected but whose shadows are not. Worse failure than this
is possible.

\begin{theorem}
\label{thm:quasi}
Given any uncertainty interval $(\epsilon,\epsilon')$ and any
finitely presented group $G$, there exists a quasi-Rips complex
$\Rips_Q$ with $\pi_1(\Rips_Q)\cong G*F$, where $F$ is a free group.
\end{theorem}
\begin{proof}
It is well-known that any finitely presented group $G$ can be
realized as the fundamental group of a 2-dimensional cell complex
whose 1-skeleton is a wedge of circles over the generators and whose
2-cells correspond to relations. Such a complex can be triangulated,
and, after a barycentric subdivision, can be assumed to be
3-colored: that is, there are no edges between vertices of the same
color. Call this vertex 3-colored 2-d simplicial complex $K$.

We perform a `blowup' of the complex $K$ to a 3-d simplicial complex
$\tilde{K}$ as follows (see Figure \ref{fig:blowup} for an example).
Recall, the geometric realization of $K$ can be expressed as the
disjoint union of closed $i$-simplices with faces glued via
simplicial gluing maps (the $\Delta$-complex \cite{Hatcher}). To
form $\tilde{K}$, take the disjoint union of closed $i$-simplices of
$K$ and instead of simplicial gluing maps, use the \df{join} to
connect all faces. The 3-coloring of $K$ is inherited by $\tilde{K}$
via the blowup process.

There is a natural collapsing map $c:\tilde{K}\to K$ which collapses
the joins to simplicial identification maps. The inverse image of
any point in an open $2$-simplex ($1$-simplex, resp.) of $K$ is a
closed $0$-simplex ($2$-simplex resp.) of $\tilde{K}$. The inverse
image of a vertex $v\in K$ consists of the 1-skeleton of the link of
$v$ in $K$. If we fill in $\tilde{K}$ by taking the flag completion,
then $c^{-1}(v)$ is a copy of the star of $v$ in $K$. Thus, upon
taking the flag complex of $\tilde{K}$, the fiber of $c$ for each
point in $K$ is contractible, which shows that the flag complex of
$\tilde{K}$ is homotopic to $K$ and thus preserves $\pi_1$.

\begin{figure}
\begin{center}
\includegraphics[width=4.25in]{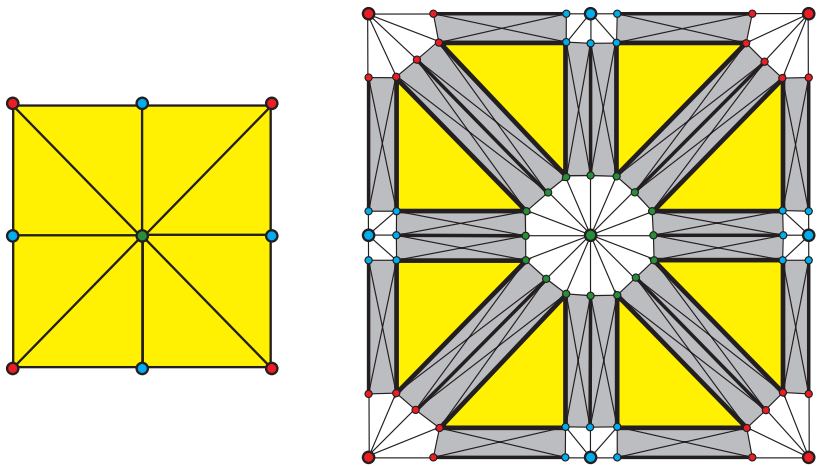}
\caption{A 3-colored simplicial complex $K$ and its blowup
$\tilde{K}$, whose flag completion is homotopy equivalent to $K$.
Opposite edges of $K$ (and thus $\tilde{K}$) can be identified to
yield a torus, projective plane, or Klein bottle.}
\label{fig:blowup}
\end{center}
\end{figure}

We now embed $\tilde{K}$ in a quasi-Rips complex $\Rips_Q$. Define
the vertices of $\Rips_Q$ in $\real^2$ as follows. Fix an
equilateral triangle of side length $(\epsilon+\epsilon')/2$ in
$\real^2$. Embed the vertices of $\tilde{K}$ arbitrarily in
sufficiently small open balls (of radii no larger than
$(\epsilon'-\epsilon)/4$) centered at the vertices of this triangle,
respecting the 3-coloring. For this vertex set in $\real^2$, we
define $\Rips_Q$ by placing an edge between vertices according to
the edges of $\tilde{K}$, using the fact that any two vertices not
of the same color are separated by a distance within the uncertainty
interval. Of course, we must also add a complete connected graph on
all vertices with a given color, since these lie within the small
balls.

The quasi-Rips complex $\Rips_Q$ is the flag complex of this graph.
It contains the flag complex of $\tilde{K}$, along with three
high-dimensional simplices, one for each color.

We claim that any 2-simplex of $\Rips_Q$ which is not also a
2-simplex of $\tilde{K}$ has all vertices of the same color. Proof:
Consider a 2-simplex $\sigma\in\Rips_Q$ spanning more than one
color. Since the only edges added to form $\Rips_Q$ from $\tilde{K}$
have both ends with identical colors, it must be that
$\sigma\cap\tilde{K}$ contains two edges which share a vertex. Any
two edges in $\tilde{K}$ which share a vertex are sent by the
collapsing map $c$ to either (1) two edges of a 2-simplex in $K$; or
(2) a single 1-simplex of $K$; or (3) a single vertex of $K$. In
either case, the entire 2-simplex $\sigma$ exists in the flag
complex of $\tilde{K}$.

We end by showing that $\pi_1(\Rips_Q)$ is a free extension of $G$.
Each of the three large colored simplices added to form $\Rips_Q$
from $\tilde{K}$ is homotopy equivalent to adding an abstract
colored vertex (the apex of the cone) and an edge from this apex to
the blowup of each $0$-simplex of $K$ in $\tilde{K}$. This is
homotopy equivalent to taking a wedge with (many) circles and thus
yields a free extension of the fundamental group of the flag complex
of $\tilde{K}$, $G$.
\end{proof}

We note that the construction above may be modified so that the
lower-bound Rips complex $\Rips_\epsilon$ is connected. If
necessary, the complex can be so constructed that the inclusion map
$\Rips_\epsilon\hookrightarrow\Rips_{\epsilon'}$ induces an
isomorphism on $\pi_1$ (which factors through $\pi_1(\Rips_Q)$).

Theorem \ref{thm:quasi} would appear to be a cause for despair,
especially for applications to sensor networks, in which the rigid
unit-disc graph assumption is unrealistic. The following result
shows that Theorem \ref{thm:main} is not without utility, even when
only quasi-Rips complexes are available.

\begin{corollary}
Let $\Rips_Q$ and $\Rips_{Q'}$ denote two quasi-Rips complexes whose
uncertainty intervals are disjoint. Then the image of
$\pi_1(\Rips_Q)$ in $\pi_1(\Rips_{Q'})$ is a free subgroup of
$\Shadow_{\epsilon'}$ for any $\epsilon'$ in between the uncertainty
intervals of the quasi-Rips complexes.
\end{corollary}

Roughly speaking, this result says that a {\em pair} of quasi-Rips
complexes, graded according to sufficiently distinct strong and weak
signal links, suffices to induce information about a shadow complex.

\begin{proof}
The inclusions $\Rips_Q\subset\Rips_{\epsilon'}\subset\Rips_{Q'}$
imply that the induced homomorphism
$\pi_1(\Rips_Q)\to\pi_1(\Rips_{Q'})$ factors through
$\pi_1(\Rips_{\epsilon'})$. Thus, the image of $\pi_1(\Rips_Q)$ in
$\pi_1(\Rips_{Q'})$ is a subgroup of $\pi_1(\Rips_{\epsilon'})\cong
\pi_1(\Shadow_{\epsilon'})$, a free group. Any subgroup of a free
group is free.
\end{proof}

This is another example of the principle of \df{topological
persistence}: there is more information in the inclusion map between
two spaces than in the two spaces themselves. Knowing two `noisy'
quasi-Rips complexes and the inclusion relating them yields true
information about the shadow.

\section{$k$-connectivity in $\real^n$}
\label{sec:kn}

Theorem \ref{thm:main}  points to the broader question of whether
higher-order topological data are preserved by the shadow projection
map. Recall that a topological space is \df{$k$-connected} if the
homotopy groups $\pi_i$ vanish for all $0\leq i \leq k$. A map
between topological spaces is $k$-connected if the induced
homomorphisms on $\pi_i$ are isomorphisms for all $0\leq i \leq k$.

We summarize the results of this section in
Figure~\ref{fig:knrips}.

\begin{figure}[hbt]
\begin{center}
\includegraphics[width=3.0in]{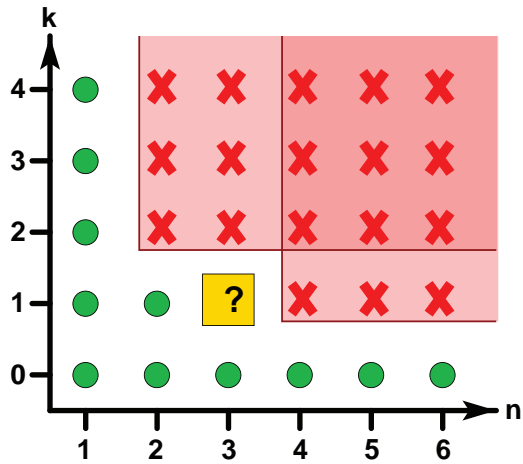}
\caption{For which $(n,k)$ is the Rips projection map in~$\euc^n$
$k$-connected? The only unresolved case is $(3,2)$.}
\label{fig:knrips}
\end{center}
\end{figure}

Throughout this paper, we have ignored basepoint considerations in
the description and computation of $\pi_1$. The following
proposition excuses our laziness.

\begin{proposition}
\label{prop:pi0}
For any set of points in~$\euc^n$, the map $\proj\colon \Rips \to
\Shadow$ is 0-connected.
\end{proposition}

\begin{proof}
Certainly $\pi_0(\proj)$ is surjective, since $\proj$ is surjective.
The injectivity of $\pi_0(\proj)$ is a consequence of the following
claim: If two Rips simplices $\sigma$ and $\tau$ have intersecting
shadows, then $\sigma$ and $\tau$ belong to the same connected
component of~$\Rips$.

To prove the claim, suppose that $\proj(\sigma)$ and $\proj(\tau)$
intersect.  By translation, we can suppose that $0 \in \proj(\sigma)
\cap \proj(\tau)$. If $\{x_i\}$ and $\{y_j\}$ respectively denote
the vertices of $\sigma$ and $\tau$, then
\[
\sum_i \lambda_i x_i = 0 = \sum_j \mu_j y_j
\]
for suitable convex coefficients $\{\lambda_i\}$ and $\{\mu_j\}$.
Then
\begin{eqnarray*}
\sum_{i,j} \lambda_i \mu_j \abs{x_i - y_j}^2
&=& \sum_{i,j} \lambda_i \mu_j \abs{x_i}^2
 - 2 \sum_{i,j} \lambda_i \mu_j (x_i \cdot y_j)
 + \sum_{i,j} \lambda_i \mu_j \abs{y_j}^2
\\
&=&
 \sum_i \lambda_i \abs{x_i}^2
 - 2 \sum_i \lambda_i x_i \cdot \sum_j \mu_j y_j
 + \sum_{j} \mu_j \abs{y_j}^2
\\
&=&
 \sum_i \lambda_i \abs{x_i}^2 + \sum_{j} \mu_j \abs{y_j}^2
 ,
\end{eqnarray*}
and similarly
\begin{eqnarray*}
\sum_{i,i'} \lambda_i \lambda_{i'} \abs{x_i - x_{i'}}^2
&=& 2 \sum_i \lambda_i \abs{x_i}^2,
\\
\sum_{j,j'} \mu_j \mu_{j'} \abs{y_j - y_{j'}}^2
&=& 2 \sum_j \mu_j \abs{y_j}^2
.
\end{eqnarray*}
Since every edge $x_i x_{i'}$ and $y_j y_{j'}$ has length at most~1,
the left-hand sides of these last equations have value at most~1.
Thus $\sum_i \lambda_i \abs{x_i}^2 \leq 1/2$ and $\sum_j \mu_j
\abs{y_j}^2 \leq 1/2$. It follows that $\sum_{i,j} \lambda_i \mu_j
\abs{x_i - y_j}^2 \leq (1/2) + (1/2) = 1$ and so at least one edge
$x_i y_j$ has length at most~1.

Thus the simplices $\sigma, \tau$ are connected by an edge, as
required.
\end{proof}

\begin{proposition}
\label{prop:1D}
For any set of points in~$\euc^1$, the map $\proj\colon\Rips \to
\Shadow$ is a homotopy equivalence.
\end{proposition}

\begin{proof}
Both $\Rips$ and $\Shadow$ are homotopy equivalent to finite unions
of closed intervals in~$\euc^1$, and therefore to finite sets of
points. This is clear for~$\Shadow$. For $\Rips$, we note that
$\Rips_1$ is equal to the \v{C}ech complex $\Cech_{1}$ in~$\euc^1$.
Certainly the two complexes have the same 1-skeleton. Moreover,
Helly's theorem implies that \v{C}ech complexes are flag complexes in
1D: a collection of convex balls has nonempty intersection if all
pairwise intersections are nonempty. Thus $\Rips_1 = \Cech_{1}$.
By the nerve theorem, this complex has the homotopy type of a union
of closed intervals in~$\euc^1$.

Since a 0-connected map between finite point sets is a homotopy
equivalence, the same conclusion now holds for the 0-connected map
$\proj\colon \Rips \to \Shadow$.
\end{proof}

\begin{proposition}
\label{prop:2D}
There exists a configuration of points in~$\euc^2$ for which $\proj$
is not 2-connected.
\end{proposition}

\begin{proof}
Consider the vertices $rx_1, rx_2, rx_3, rx_4, rx_5, rx_6$ of a
regular hexagon of radius~$r$ centered at the origin. If $1/2 < r
\leq 1/\sqrt{3}$ then only the three main diagonals are missing
from~$\Rips$. Thus $\Rips$ has the structure of a regular
octahedron, and therefore the homotopy type of a 2-sphere. On the
other hand $\Shadow$ is just the hexagon itself (including
interior), and is contractible.
\end{proof}

The example of Proposition \ref{prop:2D} extends to higher homotopy
groups by constructing cross-polytopes, as in \cite{DG:controlled}.

\begin{proposition}
\label{prop:4D}
There exists a configuration of points in~$\euc^4$ for which $\proj$
is not 1-connected.
\end{proposition}

\begin{proof}
Consider the six points
\[
(r x_1, \epsilon x_1), \quad (r x_2, 0), \quad
(r x_3, \epsilon x_3), \quad (r x_4, 0), \quad
(r x_5, \epsilon x_5), \quad (r x_6, 0)
\]
in~$\euc^4$, in the notation of the previous proposition. Then
$\Rips$ has the structure of a regular octahedron, but the map
$\proj\colon \Rips \to \Shadow$ identifies one pair of antipodal points
(specifically, the centers of the two large triangles, 135 and 246).
Thus $\Rips$ is simply-connected, whereas $\pi_1(\Shadow) =
\mathbb{Z}$.
\end{proof}

We note that these counterexamples may be embedded in higher
dimensions and perturbed to lie in general position.
%

\section{Conclusion}
\label{sec:conclusions}

The relationship between a Rips complex and its projected shadow is
extremely delicate, as evidenced by the universality result for
quasi-Rips complexes (Theorem \ref{thm:quasi}) and the lack od
general $k$-connectivity in $\real^n$ (\S\ref{sec:kn}). These
results act as a foil to Theorem \ref{thm:main}: it is by no means
{\em a priori} evident that a planar Rips complex should so
faithfully capture its shadow.

We close with a few remarks and open questions.
\begin{enumerate}
\item
Are the cross-polytopes of Proposition \ref{prop:2D} the only
significant examples of higher homology in a (planar) Rips complex?
If all generators of the homology $H_k(\Rips)$ for $k>1$ could be
classified into a few such `local' types, then, after a local
surgery on $\Rips$ to eliminate higher homology, one could use the
Euler characteristic combined with Theorem \ref{thm:main} as a means
of quickly computing the number of holes in the shadow of a planar
Rips complex. This method would have the advantage of being local
and thus distributable.
\item
Does the projection map preserve $\pi_1$ for a Rips complex of
points in $\real^3$? Our proofs for the 2-d case rest on some
technical lemmas whose extensions to 3-d would be neither easy to
write nor enjoyable to read. A more principled approach would be
desirable, but is perhaps not likely given the $1$-connectivity on
$\real^3$ is a borderline case.
\item
What are the computational and algorithmic issues associated with
determining the shadow of a (planar) Rips complex? See \cite{CEW}
for recent progress, including algorithms for test contractibility
of cycles in a planar Rips complex and a positive lower bound on the
diameter of a hole in the shadow.
\end{enumerate}


\bibliographystyle{plain}

\begin{thebibliography}{99}

\bibitem{Barriere}
L. Barri\`ere, P. Fraigniaud, and L. Narayanan, ``Robust
position-based routing in wireless ad hoc networks with unstable
transmission ranges,'' In {\em Proc. 
Workshop on Discrete Algorithms and Methods for Mobile Computing and
Communications}, 2001.

\bibitem{Leray}
A. Bj\"orner, ``Topological methods'', in {\em Handbook of
Combinatorics} (R. Graham, M. Gr\"otschel, and L. Lov\'asz, Eds.),
1819--1872, North- Holland, Amsterdam, 1995.

\bibitem{CCD}
E. Carlsson, G. Carlsson, and V. de Silva, ``An algebraic
topological method for feature identification,'' \emph{Intl. J.
Computational Geometry and Applications}, 16:4 (2006),
291-–314.

\bibitem{CIDZ}
G. Carlsson, T. Ishkhanov, V. de Silva, and A. Zomorodian, ``On the
local behavior of spaces of natural images,'' preprint, (2006).

\bibitem{CZCG} G. Carlsson, A. Zomorodian, A. Collins, and
L. Guibas, ``Persistence barcodes for shapes,'' \emph{Intl. J. Shape
Modeling}, 11 (2005), 149–-187.

\bibitem{CEW}
E. Chambers, J. Erickson, and P. Worah, ``Testing contractibility in
planar Rips complexes,'' preprint, (2007).

\bibitem{DC}
V. de Silva and G. Carlsson. ``Topological estimation using witness
complexes,'' in \emph{SPBG'04 Symposium on Point-Based Graphics}
(2004), 157–-166.

\bibitem{DG:controlled} V. de Silva and R. Ghrist,
``Coordinate-free coverage in sensor networks with controlled
boundaries via homology,'' \emph{Intl. J. Robotics Research}, 25:12,
(2006), 1205--1222.

\bibitem{DG:persistence}
V. de Silva and R. Ghrist, ``Coverage in sensor networks via
persistent homology,'' \emph{Alg. \& Geom. Top.}, 7, (2007),
339--358.

\bibitem{Helly}
J. Eckhoff, ``Helly, Radon, and Carath\'{e}odory Type Theorems.''
Ch. 2.1 in \emph{Handbook of Convex Geometry} (Ed. P. M. Gruber and
J. M. Wills). Amsterdam, Netherlands: North-Holland, pp. 389--448,
1993.

\bibitem{EM}
H. Edelsbrunner and E.P. M\"ucke, ``Three-dimensional alpha
shapes,'' \emph{ACM Transactions on Graphics}, 13:1, (1994), 43–-72.

\bibitem{FGG04}
Q. Fang and J. Gao and L. Guibas, ``Locating and Bypassing Routing
Holes in Sensor Networks,'' in Proc. 23rd Conference of the IEEE
Communications Society (InfoCom), 2004.

\bibitem{Gromov}
M. Gromov,
\emph{Hyperbolic groups}, in Essays in Group Theory, MSRI Publ.
\textbf{8}, Springer-Verlag, 1987.

\bibitem{Hatcher}
A. Hatcher, \emph{Algebraic Topology}, Cambridge University Press,
2002.

\bibitem{Kuhn}
F. Kuhn, R. Wattenhofer, and A. Zollinger, ``Ad-hoc networks beyond
unit disk graphs, in {\em Proc. Foundations of Mobile Computing},
2003.

\bibitem{Jad}
A. Muhammad and A. Jadbabaie, ``Dynamic coverage verification in
mobile sensor networks via switched higher order Laplacians,'' in
{\em Robotics: Science \& Systems}, 2007.

\bibitem{Vietoris}
L. Vietoris, ``\"Uber den h\"oheren Zusammenhang kompakter R\"aume
und eine Klasse von zusammenhangstreuen Abbildungen,'' \emph{Math.
Ann.} 97 (1927), 454--472.

\end{thebibliography}

\end{document}